\documentclass[11pt,british]{amsart}

\usepackage{babel,eucal,url,amssymb,
enumerate,amscd,
}

\usepackage[pagebackref]{hyperref}

\theoremstyle{plain}
\newtheorem{thm}{Theorem}[section]
\newtheorem{lem}[thm]{Lemma}
\newtheorem{pro}[thm]{Proposition}

\theoremstyle{definition}
\newtheorem{defn}[thm]{Definition}

\theoremstyle{remark}
\newtheorem{rem}[thm]{Remark}

\newcommand{\Gtwo}{\ifmmode{{\rm G}_2}\else{${\rm G}_2$}\fi}

\date{\today}

\begin{document}

\title[]
 {A Classification of 3-dimensional paracontact
metric manifolds with $Q\varphi=\varphi Q$}

\author[S. Zamkovoy]{Simeon Zamkovoy}

\address{
University of Sofia "St. Kl. Ohridski"\\
Faculty of Mathematics and Informatics\\
Blvd. James Bourchier 5\\
1164 Sofia, Bulgaria}
\email{zamkovoy@fmi.uni-sofia.bg}

\author[A. Bojilov]{ Assen Bojilov}

\address{
University of Sofia "St. Kl. Ohridski"\\
Faculty of Mathematics and Informatics\\
Blvd. James Bourchier 5\\
1164 Sofia, Bulgaria}
\email{bojilov@fmi.uni-sofia.bg}

\subjclass{}

\keywords{3-dimensional paracontact metric manifolds}


\begin{abstract}
We show that a $3-$dimensional paracontact manifold on which $Q\varphi =\varphi Q$
is either a manifold with $trh^2=0$, flat or of constant $\xi-$sectional curvature $k\neq-1$ and constant $\varphi$-sectional
curvature $-k\neq 1$.
\end{abstract}

\newcommand{\g}{\mathfrak{g}}
\newcommand{\s}{\mathfrak{S}}
\newcommand{\D}{\mathcal{D}}
\newcommand{\F}{\mathcal{F}}
\newcommand{\R}{\mathbb{R}}
\newcommand{\K}{\mathbb{K}}
\newcommand{\U}{\mathbb{U}}
\newcommand{\diag}{\mathrm{diag}}
\newcommand{\End}{\mathrm{End}}
\newcommand{\im}{\mathrm{Im}}
\newcommand{\id}{\mathrm{id}}
\newcommand{\Hom}{\mathrm{Hom}}

\newcommand{\Rad}{\mathrm{Rad}}
\newcommand{\rank}{\mathrm{rank}}
\newcommand{\const}{\mathrm{const}}
\newcommand{\tr}{{\rm tr}}
\newcommand{\ltr}{\mathrm{ltr}}
\newcommand{\codim}{\mathrm{codim}}
\newcommand{\Ker}{\mathrm{Ker}}

\newcommand{\thmref}[1]{Theorem~\ref{#1}}
\newcommand{\propref}[1]{Proposition~\ref{#1}}
\newcommand{\corref}[1]{Corollary~\ref{#1}}
\newcommand{\secref}[1]{\S\ref{#1}}
\newcommand{\lemref}[1]{Lemma~\ref{#1}}
\newcommand{\dfnref}[1]{Definition~\ref{#1}}


\newcommand{\ee}{\end{equation}}
\newcommand{\be}[1]{\begin{equation}\label{#1}}

\maketitle

\section{Introduction}\label{sec-1}
The assumption that $(M^{2n+1},\varphi,\xi,\eta,g)$ is a paracontact metric manifolds is very weak,
since the set of metrics associated to the paracontact form $\eta$ is huge. Even if the structure is
$\eta-$Einstein we do not have a complete classification. Also for $n=1$, we known very little about
the geometry of these manifolds (see \cite{Z1}). On the other hand if the structure is para-Sasakian, the Ricci operator
$Q$ commutes with $\varphi$ (see \cite{Z1}), but in general $Q\varphi \neq \varphi Q$ and the problem of the characterization
of paracontact metric manifolds with $Q\varphi =\varphi Q$ is open. In \cite{T} Tanno defined a special family of paracontact
metric manifolds by the requirement that $\xi$ belong to the $k-$nullity distribution of $g$. We also know very little
about these manifolds (see \cite{Z1}). In this paper, we show that a $3-$dimensional paracontact manifold on which $Q\varphi =\varphi Q$
is either a manifold with $trh^2=0$, flat or of constant $\xi-$sectional curvature $k\neq-1$ and constant $\varphi$-sectional
curvature $-k\neq 1$.

\section{Preliminaries}\label{sec-2}
A $C^{\infty}$ manifold $M^{(2n+1)}$
is said to be \emph{paracontact manifold}, if it carries a global $1-$form
$\eta$ such that $\eta\wedge(d\eta)^{n}\neq 0$ everywhere. We assume throughout that all manifolds are connected. Given a paracontact form $\eta$,
it is well known that there exists a unique vector field $\xi$, called \emph{characteristic vector field} of $\eta$, satisfying $\eta(\xi)=1$ and
$d\eta(\xi,X)=0$ for all vector fields $X$. A pseudo-Riemannian metric $g$ is said to be an \emph{associated metric} if there exists a tensor field
$\varphi$ of type $(1,1)$ such that
\begin{equation}\label{con}
d\eta(X,Y)=g(X,\varphi Y), \quad \eta(X)=g(X,\xi), \quad \varphi^2=Id-\eta\otimes \xi.
\end{equation}
From these conditions one easily obtains
\begin{equation}\label{b1}
\varphi\xi=0, \quad \eta\circ\varphi=0,\quad g(\varphi X,\varphi Y)=-g(X,Y)+\eta (X)\eta (Y).
\end{equation}
The structure $(\varphi,\xi,\eta,g)$ is called a \emph{paracontact metric structure}, and a manifold $M^{2n+1}$ with paracontact metric structure
$(\varphi,\xi,\eta,g)$ is said to be a \emph{paracontact metric manifold}.

Denoting by $\pounds$ and $R$ the Lie differentiation and the curvature tensor respectively, we define the operators $\tau$, $l$ and $h$ by
\begin{equation}\label{1}
\tau=\pounds_{\xi}g, \quad lX=R(X,\xi)\xi, \quad h=\frac{1}{2}\pounds_{\xi}\varphi.
\end{equation}
The $(1,1)-$type tensors $h$ and $l$ are symmetric and satisfy
\begin{equation}\label{2}
l\xi=0,\quad h\xi=0,\quad trh=0,\quad trh\varphi=0 \quad and \quad h\varphi=-\varphi h.
\end{equation}
We also have the following formulas for a paracontact manifold:
\begin{equation}\label{3}
\nabla_{X}\xi=-\varphi X+\varphi hX \ (and \ hence \ \nabla_{\xi}\xi=0)
\end{equation}
\begin{equation}\label{4}
\nabla_{\xi}\varphi=0
\end{equation}
\begin{equation}\label{5}
trl=g(Q\xi,\xi)=-2n+trh^2
\end{equation}
\begin{equation}\label{6}
\varphi l\varphi+l=-2(\varphi^2-h^2)
\end{equation}
\begin{equation}\label{7}
\nabla_{\xi}h=-\varphi-\varphi l+\varphi h^2,
\end{equation}
where $tr$ is the trace of the operator, $Q$ is the Ricci operator and $\nabla$ is the Levi-Civita connection of $g$. The formulas are proved in \cite{Z}.

A paracontact metric manifold for which $\xi$ is Killing is called a $K-\emph{paracontact}$ $\emph{manifold}$. A paracontact structure on $M^{(2n+1)}$ naturally gives rise to an almost paracomplex structure on the product $M^{(2n+1)}\times\Re$. If this almost paracomplex structure is integrable, the given paracontact metric manifold is said to be a $\emph{para-Sasakian}$. Equivalently, (see \cite{Z}) a paracontact metric manifold is a para-Sasakian if and only if
\begin{equation}\label{8}
(\nabla_{X}\varphi)Y=-g(X,Y)\xi+\eta(Y)X,
\end{equation}
for all vector fields $X$ and $Y$.

It is easy to see that a $3-$dimentional paracontact manifold is para-Sasakian if and only if $h=0$. For details we refer the reader to \cite{JW},\cite{Z}.

A paracontact metric structure is said to be  $\eta-\emph{Einstein}$ if
\begin{equation}\label{8.1}
Q=a.id+b.\eta\otimes\xi,
\end{equation}
where $a,b$ are smooth functions on $M^{(2n+1)}$. We also recall that the $k-$nullity distribution $N(k)$ of a pseudo-Riemannian manifold $(M,g)$, for a real number $k$, is the distribution
\begin{equation}\label{T}
N_p(k)=\{Z\in T_pM:R(X,Y)Z=k(g((Y,Z)X-g(X,Z)Y)\},
\end{equation}
for any $X,Y\in T_pM$ (see \cite{T}).

Finally, the sectional curvature $K(\xi,X)=\epsilon_XR(X,\xi,\xi,X)$, where $|X|=\epsilon_X=\pm 1$, of a plane section spanned by $\xi$ and the vector $X$ orthogonal to $\xi$ is called $\emph{$\xi$-sectional curvature}$, whereas the sectional curvature $K(X,\varphi X)=-R(X,\varphi X,\varphi X,X)$, where $|X|=-|\varphi X|=\pm 1$,
of a plane section spanned by vectors $X$ and $\varphi X$ orthogonal to $\xi$ is called a $\emph{$\varphi$-sectional curvature}$.

\section{Main result}\label{sec-4}
Before we state our main result we need the following lemma which was proved in \cite{Z1}, but we include its proof here for completeness and because we will use many of formulas which appear in the proof.
\begin{lem}\label{l1}
Let $M^3$ be a paracontact metric manifold with a paracontact metric structure $(\varphi,\xi,\eta,g)$ such that $\varphi Q=Q\varphi$. Then the function $trl$ is constant everywhere on $M^3$.
\end{lem}
\begin{proof}
We recall that the curvature tensor of a 3-dimensional pseudo-Riemannian manifold is given by
\begin{equation}\label{9}
R(X,Y)Z=g(Y,Z)QX-g(X,Z)QY+g(QY,Z)X-g(QX,Z)Y-
\end{equation}
$-\frac{scal}{2}(g(Y,Z)X-g(X,Z)Y),$

where $scal$ is the scalar curvature of the manifold.

Using $\varphi Q=Q\varphi$, \eqref{5} and $\varphi\xi=0$ we have that
\begin{equation}\label{10}
Q\xi=(trl)\xi.
\end{equation}
From \eqref{9} and using \eqref{1} and \eqref{10}, we have that for any $X$,
\begin{equation}\label{11}
lX=QX+(trl-\frac{scal}{2})X+\eta(X)(\frac{scal}{2}-2trl)\xi
\end{equation}
and hence $\varphi Q=Q\varphi$ and $\varphi\xi=0$ give
\begin{equation}\label{12}
\varphi l=l\varphi.
\end{equation}
By virtue of \eqref{12}, \eqref{6} and \eqref{7}, we obtain
\begin{equation}\label{13}
-l=\varphi^2-h^2
\end{equation}
and $\nabla_{\xi}h=0$. Differentiating \eqref{13} along $\xi$ and using \eqref{4} and $\nabla_{\xi}h=0$ we find that $\nabla_{\xi}l=0$ and therefore $\xi trl=0$. If at point $p\in M^3$ there exists
$X\in T_pM$, $X\neq\xi$ such that $lX=0$, then $l=0$ at the point $P$. In fact if $Y$ is the projection of $X$ on $\mathbb {D}$, we have $lY=0$, since $l\xi=0$. Using \eqref{12} we have $l\varphi Y=0$.
So $l=0$ at the point $P$ (and thus $trl=0$ at the point $P$). We now suppose that $l\neq 0$ on a neighborhood $U$ of the point $P$. Using \eqref{12} and that $\varphi$ is antisymmetric, we get $g(\varphi X,lX)=0$.
So $lX$ is parallel to $X$ for any $X$ orthogonal to $\xi$. It is not hard to see that $lX=\frac{trl}{2}X$ for any $X$ orthogonal to $\xi$. Thus for any $X$, we have
\begin{equation}\label{14}
lX=\frac{trl}{2}\varphi^2X
\end{equation}
Substituting \eqref{14} in \eqref{11} we get
\begin{equation}\label{15}
QX=aX+b\eta(X)\xi,
\end{equation}
where $a=\frac{scal-trl}{2}$ and $b=\frac{3trl-scal}{2}$. Differentiating \eqref{15} with respect to $Y$ and using \eqref{15} and $\nabla_{\xi}\xi=0$ we find
\begin{equation}\label{16}
(\nabla_{Y}Q)X=(Ya)X+((Yb)\eta(X)+bg(X,\nabla_Y\xi))\xi+b\eta(X)\nabla_Y\xi.
\end{equation}
So using $\xi trl=0$ and $\nabla_{\xi}\xi=0$, from  \eqref{16} with $X=Y=\xi$, we have $(\nabla_{\xi}Q)\xi=0$. Also using $h\varphi=-\varphi h$, and \eqref{3}, from \eqref{16} with $X=Y$ orthogonal to $\xi$, we get
\begin{equation}\label{17}
g((\nabla_X Q)X-(\nabla_{\varphi X}Q)\varphi X,\xi)=0.
\end{equation}
But it is well known that $$(\nabla_XQ)X-(\nabla_{\varphi X}Q)\varphi X+(\nabla_{\xi}Q)\xi=\frac{1}{2}grad(scal),$$
for any unit vector $X$ orthogonal to $\xi$. Hence, we easily get from the last two equations that $\xi(scal)=0$, and thus $\nabla_{\xi}Q=0$. Therefore, differentiating \eqref{9} with respect to $\xi$ and using $\nabla_{\xi}Q=0$,
we have $\nabla_{\xi}R=0$. So from the second identity of Bianchi, we get
\begin{equation}\label{18}
(\nabla_X R)(Y,\xi,Z)=(\nabla_{Y}R)(X,\xi,Z).
\end{equation}
Now, substituting \eqref{15} in \eqref{9}, we obtain
\begin{equation}\label{19}
R(X,Y)Z=(\gamma g(Y,Z)+b\eta(Y)\eta(Z))X-(\gamma g(X,Z)+b\eta(X)\eta(Z))Y+
\end{equation}
$$+b(\eta(X)g(Y,Z)-\eta(Y)g(X,Z))\xi,$$
where $\gamma=\frac{scal}{2}-trl$. For $Z=\xi$, \eqref{19} gives
\begin{equation}\label{20}
R(X,Y)\xi=\frac{trl}{2}(\eta(Y)X-\eta(X)Y).
\end{equation}
Using \eqref{20}, we obtain $(\nabla_X R)(Y,\xi,\xi)=\frac{X(trl)}{2}Y$, for $X,Y$ orthogonal to $\xi$. From this and \eqref{18} for $Z=\xi$, we get $(Xtrl)Y=(Ytrl)X$. Therefore $Xtrl=0$ for $X$ orthogonal to $\xi$, but $\xi(trl)=0$,
so the function $trl$ is constant and this completes the proof of the Lemma.
\end{proof}

\begin{rem}\label{r1}
When $l=0$ everywhere, then using \eqref{9}, \eqref{10} and \eqref{11} we get $R(X,Y)\xi=0$. This together with Theorem 3.3 in \cite{ZT} gives that $M^3$ is flat.
\end{rem}
We have the following
\begin{pro} \cite{Z1} \label{p2}
Let $M^3$ be a paracontact metric manifold with paracontact metric structure $(\varphi,\xi,\eta,g)$. Then the following conditions are equivalent:

i) $M^3$ is a $\eta-$Einstein

ii) $Q\varphi=\varphi Q$

iii) $\xi$ belongs to the $k-$nullity distribution
\end{pro}
Our main theorem is
\begin{thm}
Let $M^3$ be a paracontact metric manifold with paracontact metric structure $(\varphi,\xi,\eta,g)$ on which $Q\varphi=\varphi Q$. Then $M^3$ is either a manifold with $trh^2=0$, flat or of constant $\xi-$sectional curvature $k\neq-1$ and constant $\varphi$-sectional curvature $-k\neq 1$.
\end{thm}
\begin{proof}
We can easily see from the proof of $Lemma~\ref{l1}$ and $Remark~\ref{r1}$ that if $trl=0$, $l=0$ in turns out that $M^3$ is flat. We can easily see from the proof of $Lemma~\ref{l1}$ that if $k=-1$, then $trl=-2$ and using \eqref{5}, we have that $M^3$ is a manifold with $trh^2=0$.

Let a $trl\neq 0$ then by $Proposition~\ref{p2}$ and \eqref{T} we have
\begin{equation}\label{30}
R(X,Y)\xi=k(\eta(Y)X-\eta(X)Y),
\end{equation}
where $k=\frac{trl}{2}\neq-1$ and $k\neq 0$. This implies that
\begin{equation}\label{31}
(\nabla_X\varphi Y)=-g(X-hX,Y)\xi+\eta(Y)(X-hX)
\end{equation}
as was pointed out by S. Zamkovoy (\cite{Z}); in fact this is true for any 3-dimensional paracontact manifold (\cite{JW}).
Computing $R(X,Y)\xi$ from \eqref{3} we have
$$R(X,Y)\xi=-(\nabla_X \varphi)Y+(\nabla_Y \varphi)X+(\nabla_X \varphi h)Y-(\nabla_Y \varphi h)X=$$
$$=-(\nabla_X \varphi)Y+(\nabla_Y \varphi)X+(\nabla_X \varphi)hY-(\nabla_Y \varphi)hX+\varphi (\nabla_Xh)Y-\varphi (\nabla_Yh)X.$$
Then using \eqref{30} and \eqref{31} we have
$$k(\eta(Y)X-\eta(X)Y)=\eta(X)(Y-hY)-\eta(Y)(X-hX)+\varphi ((\nabla_Xh)Y-(\nabla_Yh)X).$$
or
\begin{equation}\label{32}
(k+1)(\eta(Y)X-\eta(X)Y)=-\eta(X)hY+\eta(Y)hX+\varphi ((\nabla_Xh)Y-(\nabla_Yh)X).
\end{equation}

When we have $k>-1$, then the operator $h$ is diagonalizable (see \cite{ME}). Now let $X$ be a unit eigenvector of $h$ (i.e. $|X|=\epsilon_X=\pm 1$), say $hX=\lambda X$, $X\perp \xi$.
Since $trh^2=2(k+1)$, $\lambda=\pm\sqrt{k+1}$  and hence is a constant. Setting $Y=\varphi X$, \eqref{32} yields
$$\varphi ((\nabla_Xh)\varphi X-(\nabla_{\varphi X}h)X)=0$$
from which
\begin{equation}\label{33}
\varphi (-\lambda\nabla_X\varphi X-h\nabla_X\varphi X-\lambda\nabla_{\varphi X}X+h\nabla_{\varphi X}X)=0.
\end{equation}
Takeing the inner product of \eqref{33} with $X$ and recalling that $\varphi h+h\varphi=0$, we have
$$\lambda g(\nabla_{\varphi X}X,\varphi X)=0.$$
Since $\lambda\neq 0$ $(k>-1)$ and $X$ is unit, $\nabla_{\varphi X}X$ is orthogonal to both $X$ and $\varphi X$ and hence collinear with $\xi$. Now
$$\eta(\nabla_{\varphi X}X)=g(\nabla_{\varphi X}X,\xi)=-g(X,\nabla_{\varphi X}\xi)=-g(-X-hX,X)=\epsilon_X(\lambda+1).$$
Therefore
$$\nabla_{\varphi X}X=\epsilon_X(\lambda+1)\xi.$$
Similarly taking the inner product of \eqref{33} with $\varphi X$ yields
$$\nabla_{X}\varphi X=\epsilon_X(\lambda-1)\xi$$
and in turn $\nabla_{X}X=0$ and
$$[X,\varphi X]=-2\epsilon_X\xi.$$
Now from the form of the curvature tensor \eqref{19}, we have
$$R(X,\varphi X)\varphi X=-\epsilon_X(\frac{scal}{2}-trl)\varphi X$$
and by direct computation using $\nabla_X\xi=(\lambda-1)\varphi X$,
$$R(X,\varphi X)X=\nabla_X\nabla_{\varphi X}X-\nabla_{\varphi X}\nabla_XX-\nabla_{[X,\varphi X]}X=$$
$$=\epsilon_X(\lambda+1)\nabla_X\xi+2\epsilon_X\nabla_{\xi}X=\epsilon_X({\lambda}^2-1)\nabla_X\xi+2\epsilon_X\nabla_{\xi}X.$$
Thus
$$\nabla_{\xi}X=(\frac{{\lambda}^2-1}{2}-\frac{scal}{4})\varphi X$$
and hence
$$[\xi,X]=(\frac{(\lambda-1)^2}{2}-\frac{scal}{4})\varphi X.$$
Now computing $R(\xi,X)\xi$, by (\eqref{30}) and by direct computation, we have
$$-({\lambda}^2-1)X=\nabla_{\xi}(-\varphi X+\varphi hX)-\nabla_{(\frac{(\lambda-1)^2}{2}-\frac{scal}{4})\varphi X}\xi=$$
$$=(\lambda-1)\varphi\nabla_{\xi}X+(\frac{(\lambda-1)^2}{2}-\frac{scal}{4})(X+hX)=$$
$$=((\lambda-1)^2(\lambda+1)-\lambda\frac{scal}{2})X$$
from which
$$scal=2({\lambda}^2-1)=2k.$$
From \eqref{30} and \eqref{19} we see that
$$K(X,\xi)=k \quad and \quad K(X,\varphi X)=-k$$
as desired.

When we have $k<-1$, then the operator $\varphi h$ is diagonalizable (see \cite{ME}). Now let $X$ be a unit eigenvector of $\varphi h$ (i.e. $|X|=\epsilon_X=\pm 1$), say $\varphi hX=\lambda X$, $X\perp \xi$.
Since $trh^2=2(k+1)$, $\lambda=\pm\sqrt{-(k+1)}$  and hence is a constant. Setting $Y=\varphi X$, \eqref{32} yields
$$(\nabla_X\varphi h)\varphi X-(\nabla_{\varphi X}\varphi h)X=0$$
from which
\begin{equation}\label{35}
-\lambda\nabla_X\varphi X-\varphi h\nabla_X\varphi X-\lambda\nabla_{\varphi X}X+\varphi h\nabla_{\varphi X}X=0.
\end{equation}
Takeing the inner product of \eqref{35} with $\varphi X$ and recalling that $\varphi h+h\varphi=0$, we have
$$\lambda g(\nabla_{\varphi X}X,\varphi X)=0.$$
Since $\lambda\neq 0$ $(k<-1)$ and $X$ is unit, $\nabla_{\varphi X}X$ is orthogonal to both $X$ and $\varphi X$ and hence collinear with $\xi$. Now
$$\eta(\nabla_{\varphi X}X)=g(\nabla_{\varphi X}X,\xi)=-g(X,\nabla_{\varphi X}\xi)=-g(-\varphi^2X+\varphi h\varphi X,X)=\epsilon_X.$$
Therefore
$$\nabla_{\varphi X}X=\epsilon_X\xi.$$
Similarly taking the inner product of \eqref{35} with $X$ yields
$$\nabla_{X}\varphi X=-\epsilon_X\xi$$
and in turn $\nabla_{X}X=-\epsilon_X\lambda\xi$ and
$$[X,\varphi X]=-2\epsilon_X\xi.$$
Now from the form of the curvature tensor \eqref{19}, we have
$$R(X,\varphi X)\varphi X=-\epsilon_X(\frac{scal}{2}-trl)\varphi X$$
and by direct computation using $\nabla_X\xi=-\varphi X+\lambda X$,
$$R(X,\varphi X)X=\nabla_X\nabla_{\varphi X}X-\nabla_{\varphi X}\nabla_XX-\nabla_{[X,\varphi X]}X.$$
Thus
$$\nabla_{\xi}X=-(\frac{{\lambda}^2+1}{2}+\frac{scal}{4})\varphi X$$
and hence
$$[\xi,X]=-\lambda X-(\frac{{\lambda}^2-1}{2}+\frac{scal}{4})\varphi X.$$
Now computing $R(\xi,X)\xi$, by (\eqref{30}) and by direct computation, we have
$$({\lambda}^2+1)X=\nabla_{\xi}(-\varphi X+\lambda X)+\lambda \nabla_{X}\xi+(\frac{{\lambda}^2-1}{2}+\frac{scal}{4})\nabla_{\varphi X}\xi=$$
$$=-\varphi\nabla_{\xi}X+\lambda\nabla_{\xi}X+\lambda\nabla_X\xi+(\frac{{\lambda}^2-1}{2}+\frac{scal}{4})(-X-\lambda\varphi X)=$$
$$=(\lambda^2+1)X+(-\lambda-2\lambda(\frac{scal}{4}+\frac{\lambda^2}{2}))\varphi X$$
from which
$$scal=-2({\lambda}^2+1)=2k.$$
From \eqref{30} and \eqref{19} we see that
$$K(X,\xi)=k \quad and \quad K(X,\varphi X)=-k$$
as desired.
\end{proof}

\begin{defn}
A paracontact metric structure $(\varphi,\xi,\eta,g)$ is said to be $\emph{locally $\varphi-$}$ $\emph{symmetric}$ if $\varphi^2(\nabla_WR)(X,Y,Z)=0$, for all vector
fields $W,X,Y,Z$ orthogonal to $\xi$.
\end{defn}

In \cite{Z1} it is proved the following
\begin{thm}\label{th1}
Let $M^3$ be a paracontact metric manifold with $Q\varphi=\varphi Q$. Then $M^3$ is locally $\varphi-$symmetric if and only if the scalar curvature $scal$ of $M^3$ is constant.
\end{thm}

\begin{rem}
Using \eqref{16} with $trl=const.$, we obtain the following formula
\begin{equation}\label{26}
2|\nabla Q|^2=|gradscal|^2-(3trl-scal)^2(4+trl)
\end{equation}
which is valid on any paracontact metric manifold $M^3$ with $Q\varphi=\varphi Q$.
\end{rem}
From  $Theorem~\ref{th1}$, we get that a locally $\varphi-$symmetric paracontact metric manifold $M^3$ is a manifold with either $scal=3trl$, $scal=-12$ or $trl=-4$.
\section*{Acknowledgments}

S.Z. is partially supported by Contract  DN 12/3/12.12.2017 and Contract 80-10-31/10.04.2019 with the Sofia University "St.Kl.Ohridski".

A.B. is partially supported by Contract 80-10-209/17.04.2019 with the Sofia University "St.Kl.Ohridski".


\begin{thebibliography}{20}



\bibitem{ME}
B. Cappelletti-Montano, I. K\"upeli Erken, C. Murathan, Nullity conditions in paracontact geometry,
Diff. Geom. Appl. 30 (2012), 665-693.





\bibitem{T}
S. Tanno, \emph{Ricci curvatures of contact Riemannian manifolds}, Tohoku Math. J. {\bf 40}:3, 441-448 (1988).

\bibitem{JW}
J.Welyczko, \emph{Para-CR Structures on almost Paracontact Metric Manifolds},
Result. Math. {\bf54}, 377-387, (2009).

\bibitem{Z}
S. Zamkovoy, {\em Canonical connections on paracontact manifolds}, Ann Glob Anal Geom. {\bf 36}, 37-60, (2009).

\bibitem{Z1}
S. Zamkovoy, {\em Notes on a class of paracontact metric 3-manifolds}, arXiv:1707.05248, (2017).


\bibitem{ZT}
S. Zamkovoy, V. Tzanov, {\em Non-existence of flat paracontact metric structures in dimension greater than or equal to five}, Annuaire de l'universite de Sofia "St. Kl. Ohridski" faculte de mathematiques et informatique {\bf 100}, 27-34, (2011).

\end{thebibliography}
\end{document}